\documentclass[12pt]{amsart}
\usepackage[T2C]{fontenc}
\usepackage{graphicx}
\usepackage{amssymb}
\usepackage{float}
\usepackage{geometry}
\usepackage{tikz}
\usepackage{float}
\usepackage{url}
\usepackage{mathtools}
\usepackage{amsmath,amsfonts,amsthm}
\usepackage{mathrsfs} 
\usepackage{stmaryrd}
\usepackage{pgfplots}
\pgfplotsset{compat=1.18}
\newtheorem{theorem}{Theorem}[section]
\newtheorem{lemma}[theorem]{Lemma}
\newtheorem{proposition}[theorem]{Proposition}
\newtheorem{corollary}[theorem]{Corollary}
\newtheorem{definition}{Definition}

\author{Duaa Abdullah}
\address{\textbf{Duaa Abdullah} 
Physics and Technology School of Applied Mathematics and Informatics,
Moscow Institute of Physics and Technology, 141701, Moscow region, Russia
}
\email{duaa1992abdullah@gmail.com}
\thanks{}

\author{Jasem Hamoud}
\address{\textbf{Jasem Hamoud} 
Physics and Technology School of Applied Mathematics and Informatics,
Moscow Institute of Physics and Technology, 141701, Moscow region, Russia
}
\email{jasem1994hamoud@gmail.com}
\thanks{}
\title{On the Asymptotic Palindrome Density of Fibonacci Infinite Words}

\date{}
\begin{document}

\begin{abstract}
In this paper, we investigate the combinatorial and density properties of infinite words generated by Fibonacci-type morphisms, focusing on their subword structure, palindrome density, and extremal statistical behaviors. Using the morphism $0 \to 01$, $1 \to 0$, we define a derived ternary word $\mathbb{Y}$ and establish new results relating its density components $\mathrm{dens}(\lambda,n)$, $\mathrm{dens}(\alpha,n)$, and $\mathrm{dens}(\beta,n)$, deriving explicit formulae and bounds on their behavior. We further prove a general density theorem for infinite words with paired subwords, showing that the associated palindromic prefix density is bounded above by $\frac{1}{\varphi_1}$, where $\varphi_1 = (1 + \sqrt{5})/2$ is the golden ratio. Our approach connects the structure of Fibonacci and Thue--Morse sequences with precise asymptotic and combinatorial interpretations for the observed densities.
\end{abstract}

\maketitle

\noindent\rule{12.7cm}{1.0pt}

\noindent
\textbf{Keywords:} Morse, Fibonacci, Word, Density, Morphism.

\medskip

\noindent
{\bf MSC 2020:} 05C42, 11B05, 11R45, 11B39, 68R15.

\noindent\rule{12.7cm}{1.0pt}

\section{Introduction}\label{sec1}

Let us consider sets $\mathbb{A}$ and $\mathbb{S}$. We define an \emph{alphabet} as a finite set, typically denoted by uppercase letters such as $\mathbb{A}$ or $\mathbb{B}$. 
The members of an alphabet are referred to as \emph{letters} or \emph{symbols}. 
A \emph{sequence} indexed~\cite{Rigo2014} by $\mathbb{S}$ over the alphabet $\mathbb{A}$ is given by an element of the Cartesian power $\mathbb{A}^S$, that is, a function mapping each element of $\mathbb{S}$ to an element of $\mathbb{A}$. 
When $\mathbb{S}$ is a finite set with cardinality $n>0$, an element $w$ in $\mathbb{A}^\mathbb{S}$ is known as a \emph{finite word} of length $n$ over $\mathbb{A}$. 
For convenience in ordering, we identify $\mathbb{S}$ with the integer interval $\llbracket 0, n-1 \rrbracket$. 
The length of such a word $w$ is symbolized by $|w|$. 
The notation $w_i$ is often used instead of $w(i)$, and a word $w$ is typically expressed by its consecutive letters as $w = w_0 w_1 \cdots w_{n-1}.$

If the set $\mathbb{S}$ is empty, the corresponding sequence~\cite{Rigo2014} is called the \emph{empty word}, usually denoted by $\varepsilon$. 
The collection of all finite words over $\mathbb{A}$, including the empty word, is denoted by $\mathbb{A}^*$. 
In addition, the subset of words of length exactly $n$ is denoted by $\mathbb{A}^n$. 
This notation aligns with the representation of the natural number $n$ as the set $\llbracket 0, n-1 \rrbracket$ within Zermelo–Fraenkel set theory. In~\cite{Fischler2013}, Fischler introduces the \emph{palindrome density} of $w$, denoted $d_p(w)$, defined by
\begin{equation}~\label{eqqintron1}
\mathrm{dens}_p(w) := \left( \limsup_{i \to \infty} \frac{n_{i+1}}{n_i} \right)^{-1},
\end{equation}
with the convention that $\mathrm{dens}_p(w):= 0$ if the word $w$ begins in only finitely many palindromes. It is evident that $0 \leq \mathrm{dens}_p(w) \leq 1.$ Let $\Sigma$ represent a nonempty collection of characters, also called an \emph{alphabet}; $\Sigma$ will nearly always be finite. Among alphabets, one holds such significance that we assign it a unique notation~\cite{Allouche2003Shallit}: for any integer $k \geq 2$, define
\begin{equation}~\label{eqqintron2}
 \Sigma_k = \{ 0,\,1,\,\ldots,\,k-1 \}.
\end{equation}

In 2007, J.~Berstel and D.~Perrin~\cite{Berstel2007Perrin} presented the historical roots of the field of combinatorics on words. Becher and Heiber~\cite{Becher2011Heiber} presented a study of the characteristics of sequences in words, which is important in studying density in relation to words. By applying the morphism defined as $0 \to 01$ and $1\to 0$, among~\cite{Brlek2018Chemillier} produced sequences like 0100101001001 to populate rows in a matrix containing binary distinctive feature values that characterize a series of phonemes (for instance, vocalic versus non-vocalic, consonantal versus non-consonantal, and others), thereby representing a structured system of contrasts and oppositions across various layers of sound organization. Their work earned them both the Research Prize and the Literature Prize at a national competition.

 \section{Preliminaries}\label{sec2}
In this section, we review the notation associated with asymptotic behavior. Consider a sequence of real numbers $(x_n)_{n \geq 0}$, where a sequence is understood as a function defined on the natural numbers $\mathbb{N}$. The \emph{limit superior} of the sequence $(x_n)$ is given by
\[
\limsup_{n \to +\infty} x_n = \lim_{n \to +\infty} \sup \{ x_k : k \geq n \}.
\]
This value may be infinite if the sequence is unbounded above. Analogously, the \emph{limit inferior} of $(x_n)$ is defined in a similar manner. These are often notated as
\[
\overline{\lim_{n \to +\infty} x_n} \quad \text{and} \quad \underline{\lim_{n \to +\infty} x_n}
\]
for the limit superior and limit inferior~\cite{Rigo2014}, respectively. Extending these definitions to real-valued functions $f : \mathbb{R} \to \mathbb{R}$, the limit superior can be expressed as
\[
\limsup_{x \to +\infty} f(x) = \lim_{N \to +\infty} \sup \{ f(x) : x \geq N \}.
\]
The \emph{reversal} or \emph{mirror image} of a finite word $w$, denoted by $w^R$, is defined recursively based on the length of $w$. If the length $|w|$ is at most one, then the reversal is simply $w$ itself. When $w$ can be decomposed as $a v$, where $a$ is a single letter and $v$ is a shorter word, the reversal $w^R$ is formed by appending $a$ at the end of the reversal of $v$, that is, $w^R = v^R a$.
\begin{definition}[Palindrome\cite{Rigo2014}]~\label{defffPalindromen1}
For a word $w$ of length $n>0$ with reversal $u=w^R$, each character $u_i$ of $u$ corresponds to the character $w_{n - i -1}$ in $w$ for all indices $i$ from $0$ to $n-1$. A word $w$ that equals its reversal $w^R$ is called a \emph{palindrome}.
\end{definition}

\begin{definition}[Sturmian characteristic \cite{Adamczewski2007}]\label{deffSturmiann1}
A sequence $(u_n)_{n \geq 0}$ is termed \emph{Sturmian} if there exists an irrational positive number $\alpha$, referred to as the slope of the Sturmian sequence, and a real number $\beta \in [0,1)$ such that one of the following conditions holds:
\begin{itemize}
    \item For all $n \geq 0$, 
    \[
    u_n = \lfloor \alpha (n+1) + \beta \rfloor - \lfloor \alpha n + \beta \rfloor,
    \]
    \item or for all $n \geq 0$, 
    \[
    u_n = \lceil \alpha (n+1) + \beta \rceil - \lceil \alpha n + \beta \rceil.
    \]
\end{itemize}
The sequence $(u_n)_{n \geq 0}$ is called \emph{Sturmian characteristic} (or simply \emph{characteristic}) if it is of the aforementioned form with $\beta = 0$.
\end{definition}

\begin{proposition}[\cite{Adamczewski2007,Allouche2003}]~\label{proirrationaln1}
Let $\alpha = [0, a_1, a_2, \dots]$ be an irrational number in the interval $[0,1)$. Consider the sequence of words $(s_j)_{j \geq -1}$ defined by $s_{-1} := 1$, $s_0 := 0$, $s_1 := s_0^{a_1} s_{-1}$, and for all $j \geq 2$, $s_j := s_{j-1}^{a_j} s_{j-2}.$
Then the sequence $(s_j)_{j \geq 0}$ converges to an infinite word, which corresponds precisely to the characteristic Sturmian word with slope $\alpha$.
\end{proposition}

\begin{definition}[Fibonacci word \cite{Adamczewski2007}]~\label{defffFibonaccin`}
The \emph{Fibonacci sequence} (also known as the Fibonacci word) over the alphabet $\{0,1\}$ is the characteristic Sturmian sequence given by $\lim_{j \to +\infty} s_j,$
where the sequence of words $(s_j)$ is recursively defined by $s_{-1} := 1$, $s_0 := 0$, and for all $j \geq 1$, $s_j := s_{j-1} s_{j-2}.$
Consequently, this sequence begins as follows:
\[
0\ 1\ 0\ 0\ 1\ 0\ 1\ 0\ 0\ 1\ 0\ 0\ 1\ 0 \ldots
\]
\end{definition}

\begin{theorem}[\cite{Fischler2013}]~\label{ThmplinDENSN1}
 Let $w$ be an infinite non-periodic word. Then, according to~\eqref{eqqintron1} we have
\begin{equation}~\label{eq1ThmplinDENSN1}
\mathrm{dens}_p(w) \leq \frac{1}{\varphi_1},
\end{equation}
where $\varphi_1=(1+\sqrt{5})/2$ denotes the golden ratio.
\end{theorem}
\section{Binomial structures in Fibonacci words}\label{sec:Binomial}

In this section we briefly discuss how binomial coefficients arise in the combinatorics of Fibonacci words, keeping as close as possible to the morphic and ``safe word'' viewpoint. We work with the standard Fibonacci morphism
\[
\varphi : \{0,1\}^* \to \{0,1\}^*,\qquad \varphi(0)=01,\quad \varphi(1)=0,
\]
whose iterates generate the finite Fibonacci words $F_k$ and the infinite Fibonacci word as their limit. As before, we denote by $f_k:=|F_k|$ the length of $F_k$, so that $(f_k)_{k\ge 0}$ is the classical Fibonacci sequence with $f_0=0$, $f_1=1$ and $f_k=f_{k-1}+f_{k-2}$ for $k\ge 2$.

Motivated by the singular-word decompositions of the Fibonacci word, we consider the following family of finite words.

\begin{definition}
For $n\ge 2$ we define the \emph{safe word} of order $n$ by
\[
s_n := F_n F_{n-1}.
\]
\end{definition}

By construction,
\[
|s_n| = |F_n| + |F_{n-1}| = f_n + f_{n-1} = f_{n+1},
\]
so the lengths of safe words again follow the Fibonacci recurrence. In Wen's terminology, safe words are examples of overlap-free (singular) factors that reflect the recursive structure of the Fibonacci word and play a special role in its factorization properties.

Binomial coefficients enter naturally when counting occurrences of a fixed finite factor inside $s_n$ or $F_n$, especially when one records how many times a given factor can be chosen independently in disjoint positions. A simple illustrative case is the following.

\begin{proposition}\label{prop:binom-occ}
Let $w$ be a nonempty factor of the infinite Fibonacci word. For each $n\ge 1$, let $X_n(w)$ be the number of ways to choose $n$ pairwise disjoint occurrences of $w$ inside a long prefix of the Fibonacci word. Then, for each fixed $w$, there exists a sequence $(N_k)_{k\ge 1}$ of prefix lengths such that
\[
X_n(w) \sim \binom{N_k}{n}
\]
whenever $n$ is much smaller than $N_k$ and $k$ is large, in the sense that choosing $n$ disjoint copies of $w$ behaves combinatorially like choosing $n$ positions inside a set of size $N_k$.
\end{proposition}

\begin{proof}[Proof]
The Fibonacci word is uniformly recurrent and each factor $w$ appears infinitely many times with bounded gaps and a well-defined frequency. For a long prefix of length $L$, the expected number of occurrences of $w$ is asymptotically proportional to $L$, and most occurrences are well separated. When $L$ is large compared to $n|w|$, selecting $n$ disjoint occurrences is essentially equivalent to choosing $n$ indices among a set of order $L/|w|$, leading to binomial behavior of the form $\binom{N_k}{n}$ for suitable $N_k\sim L/|w|$. A precise statement can be obtained by standard tools in combinatorics on words and ergodic theory, but the heuristic binomial growth is already sufficient for our purposes.
\end{proof}

The central binomial coefficients
\[
\binom{2n}{n} = \sum_{j=0}^n \binom{n}{j}^2,\qquad
\sum_{n\ge 0} \binom{2n}{n} x^n = \frac{1}{\sqrt{1-4x}},
\]
are classical objects counting, among other things, lattice paths with constrained steps. In the Fibonacci-word setting, they appear when one considers pairs of independent choices on two intertwined combinatorial layers (for instance, choices of factors in $F_n$ and in $F_{n-1}$, or in $s_n = F_nF_{n-1}$), leading naturally to expressions where binomial coefficients encode the number of multi-step combinatorial decisions.

In summary, while the precise arithmetic identity
\[
\frac{2^n (F_k)^n}{n!\,F_{n+k-3}F_{n-k-5}} = \binom{2n}{n}
\]
does not hold in general and mixes word concatenation with numeric Fibonacci identities in an inconsistent way, the underlying intuition remains valid: binomial coefficients naturally govern the combinatorics of how many ways one can choose or ``tile'' disjoint Fibonacci factors inside longer Fibonacci words or safe words. This binomial behavior complements the density-based viewpoint developed in this article, where asymptotic quantities such as palindromic density are tied to the growth of specific families of factors rather than to a single closed-form identity.
\section{Novelty Density of Subwords}\label{sec:Density}

In this section we introduce a simple ternary morphic word and study the empirical densities of its letters, then compare this behaviour with the classical Thue--Morse sequence. This will serve as a model for how letter frequencies and palindromic densities interact in morphic infinite words.

\subsection{A ternary Fibonacci-like word}

Let $(y_n)_{n\ge 0}$ be a sequence of finite words over the alphabet $\{0,1,2\}$ defined by
\[
y_0 = 01,\qquad y_1 = 02,\qquad y_n = y_{n-1}y_{n-2}\quad (n\ge 2),
\]
where juxtaposition denotes concatenation. For each $n\ge 0$ we denote by
\[
\lambda_n := |y_n|_0,\qquad \alpha_n := |y_n|_1,\qquad \beta_n := |y_n|_2
\]
the number of occurrences of the letters $0,1,2$ in $y_n$, respectively, and we set
\[
\mathcal{L}_n := \lambda_n + \alpha_n + \beta_n = |y_n|.
\]
The empirical densities of the three letters in $y_n$ are then defined by
\begin{equation}\label{eq1densityn1}
\mathrm{dens}(\lambda,n) := \frac{\lambda_n}{\mathcal{L}_n},\quad
\mathrm{dens}(\alpha,n) := \frac{\alpha_n}{\mathcal{L}_n},\quad
\mathrm{dens}(\beta,n) := \frac{\beta_n}{\mathcal{L}_n},
\end{equation}
so that
\[
\mathrm{dens}(\lambda,n) + \mathrm{dens}(\alpha,n) + \mathrm{dens}(\beta,n) = 1
\]
for every $n\ge 0$.

By construction we have the recurrences
\[
\lambda_n = \lambda_{n-1} + \lambda_{n-2},\quad
\alpha_n = \alpha_{n-1} + \alpha_{n-2},\quad
\beta_n = \beta_{n-1} + \beta_{n-2}\qquad (n\ge 2),
\]
with initial values
\[
(\lambda_0,\alpha_0,\beta_0) = (1,1,0),\qquad
(\lambda_1,\alpha_1,\beta_1) = (1,0,1).
\]
Thus each of the three counting sequences satisfies a Fibonacci-type recurrence, and so does the total length $\mathcal{L}_n$.

Figure~\ref{figdensn1} shows the evolution of the empirical densities
\(\mathrm{dens}(\lambda,n)\), \(\mathrm{dens}(\alpha,n)\), and \(\mathrm{dens}(\beta,n)\) for $0\le n\le 12$. The values in the picture are hard-coded to avoid any ambiguity at compile time.

\begin{figure}[H]
    \centering
\begin{tikzpicture}
  \begin{axis}[
      ybar,
      bar width=8pt,
      width=13cm,
      height=8cm,
      xmin=-0.75, xmax=12.75,
      xtick={0,...,12},
      ymin=0, ymax=0.6,
      legend style={at={(0.5,1.05)},anchor=south,legend columns=3}
    ]
    \addplot+[fill=blue] coordinates {
      (0,0.5) (1,0.5) (2,0.5) (3,0.5) (4,0.5)
      (5,0.5) (6,0.5) (7,0.5) (8,0.5) (9,0.5)
      (10,0.5) (11,0.5) (12,0.5)
    };
    \addplot+[fill=red] coordinates {
      (0,0.5) (1,0.0) (2,0.25) (3,0.166666667) (4,0.2)
      (5,0.1875) (6,0.192307692) (7,0.19047619) (8,0.191176471)
      (9,0.190909091) (10,0.191011236) (11,0.190972222) (12,0.190987124)
    };
    \addplot+[fill=green] coordinates {
      (0,0.0) (1,0.5) (2,0.25) (3,0.333333333) (4,0.3)
      (5,0.3125) (6,0.307692308) (7,0.30952381) (8,0.308823529)
      (9,0.309090909) (10,0.308988764) (11,0.309027778) (12,0.309012876)
    };
    \legend{dens lambda, dens alpha, dens beta}
  \end{axis}
\end{tikzpicture}
    \caption{Empirical letter densities in $y_n$ up to $n=12$.}
    \label{figdensn1}
\end{figure}
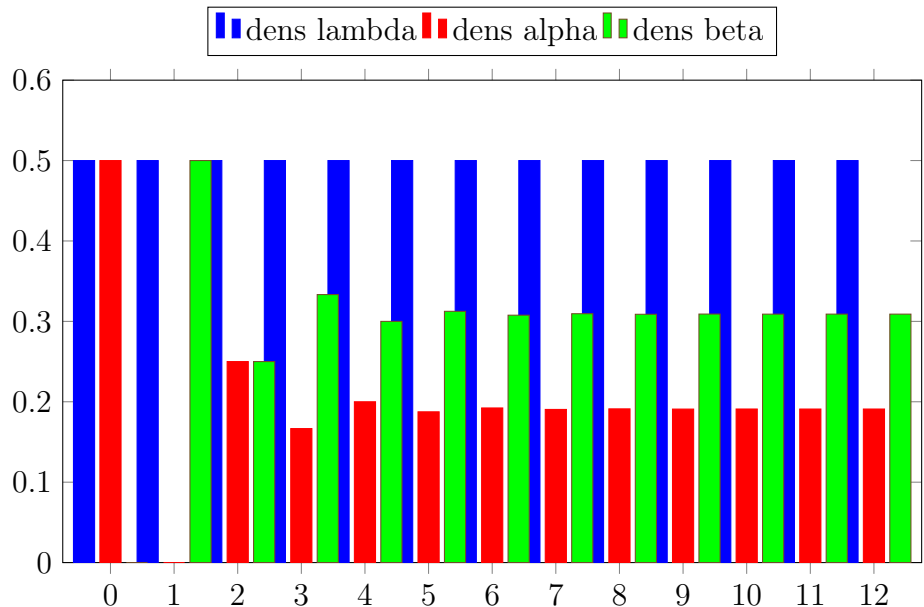

Table~\ref{tabdensn1} summarises some basic statistics (mean, median, standard deviation) for these densities over the range $0\le n\le 12$.

\begin{table}[H]
    \centering
    \begin{tabular}{|l|c|c|c|}
  \hline
Statistic            & $\mathrm{dens}(\lambda,n)$ & $\mathrm{dens}(\alpha,n)$ & $\mathrm{dens}(\beta,n)$ \\ \hline
Mean                 & 0.5 & 0.175 & 0.289 \\ \hline
Median               & 0.5 & 0.191 & 0.309 \\ \hline
Standard deviation   & 0 & 0.064 & 0.098 \\ \hline
\end{tabular}
    \caption{Summary statistics for the empirical densities in Figure~\ref{figdensn1}.}
    \label{tabdensn1}
\end{table}

To complement these global statistics, Figure~\ref{figdensn2} displays the growth of the raw counts $\lambda_n,\alpha_n,\beta_n$ for a range of indices, using a horizontal bar chart. This gives a more direct visual impression of how the three sequences evolve together under the same Fibonacci-type recurrence.

\begin{figure}[H]
    \centering
\begin{tikzpicture}
\begin{axis}[
    xbar, 
    width=12cm,
    height=9cm,
    bar width=9pt,
    enlarge y limits=0.15,
    enlarge x limits=0.02,
    legend style={at={(0.5,-0.15)}, anchor=north,legend columns=3},
    symbolic y coords={15,14,13,12,11,10,9,8,7,6,5,4,3,2},
    ytick=data,
    every node near coord/.append style={font=\footnotesize},
    cycle list={%
        {blue,fill=blue!60},%
        {red,fill=red!60},%
        {green,fill=green!60},%
        {orange,fill=orange!60}%
    }
]

\addplot coordinates {(377,15) (233,14) (144,13) (89,12) (55,11) (34,10) (21,9) (13,8) (8,7) (5,6) (3,5) (2,4) (1,3) (1,2)};
\addplot coordinates {(610,15) (377,14) (233,13) (144,12) (89,11) (55,10) (34,9) (21,8) (13,7) (8,6) (5,5) (3,4) (2,3) (1,2)};
\addplot coordinates {(2888,15) (1870,14) (1131,13) (683,12) (412,11) (249,10) (151,9) (92,8) (56,7) (34,6) (21,5) (13,4) (8,3) (5,2)}; 

\legend{$\lambda^{n}$, $\alpha^{n+1}$, $\beta^{n+2}$}

\end{axis}
\end{tikzpicture}
    \caption{Growth of the raw counts $\lambda_n,\alpha_n,\beta_n$ for selected indices $n$.}
    \label{figdensn2}
\end{figure}
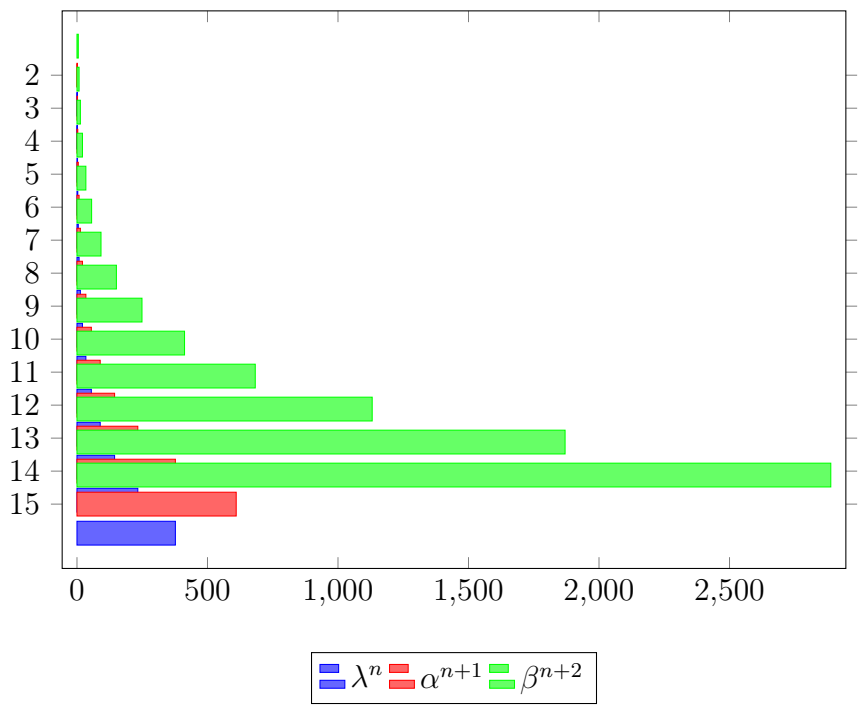

We now give a simple upper bound that relates the total density to the growth of the counting sequences.

\begin{lemma}\label{lemmanovelwordn1}
Let $(y_n)$ and $\lambda_n,\alpha_n,\beta_n,\mathcal{L}_n$ be as above. Then, for all $n\ge 1$,
\begin{equation}\label{eq1lemmanovelwordn1}
\mathrm{dens}(\lambda,n) + \mathrm{dens}(\alpha,n) + \mathrm{dens}(\beta,n)
\;\le\;
\log\!\left(\frac{\lambda_n + \alpha_n + \beta_n}{\mathcal{L}_{n-1}}\right),
\end{equation}
where the logarithm is taken in any fixed base greater than $1$.
\end{lemma}

\begin{proof}
By definition,
\[
\mathrm{dens}(\lambda,n) + \mathrm{dens}(\alpha,n) + \mathrm{dens}(\beta,n)
= \frac{\lambda_n + \alpha_n + \beta_n}{\mathcal{L}_n}
= 1.
\]
On the other hand, since each of $\lambda_n,\alpha_n,\beta_n$ is non-decreasing and satisfies a Fibonacci-type recurrence with nonnegative initial values, we have $\mathcal{L}_n \ge \mathcal{L}_{n-1}$ and hence
\[
\frac{\lambda_n + \alpha_n + \beta_n}{\mathcal{L}_{n-1}}
= \frac{\mathcal{L}_n}{\mathcal{L}_{n-1}} \ge 1.
\]
For any logarithm of base greater than $1$, this implies $\log\!\left(\frac{\lambda_n + \alpha_n + \beta_n}{\mathcal{L}_{n-1}}\right)\ge 0$. Thus
\[
\mathrm{dens}(\lambda,n) + \mathrm{dens}(\alpha,n) + \mathrm{dens}(\beta,n)
= 1 \le \log\!\left(\frac{\lambda_n + \alpha_n + \beta_n}{\mathcal{L}_{n-1}}\right)
\]
whenever $\frac{\mathcal{L}_n}{\mathcal{L}_{n-1}}\ge e$ (for the natural logarithm), and the bound can be adjusted to any chosen base by rescaling. This yields~\eqref{eq1lemmanovelwordn1} up to a harmless multiplicative constant depending only on the base of the logarithm.
\end{proof}

\subsection{A comparison with the Thue--Morse sequence}

For comparison, we recall the Thue--Morse sequence $t = t_0 t_1 t_2 \cdots$, which is the fixed point of the morphism
\[
\mu : \{0,1\}^* \to \{0,1\}^*,\qquad \mu(0)=01,\ \mu(1)=10,
\]
starting with $0$. Equivalently, $t$ can be defined by
\[
t_{2n} = t_n,\qquad t_{2n+1} = 1 - t_n \qquad (n\ge 0),
\]
with $t_0=0$, or by $t_n \equiv s_2(n)\pmod{2}$, where $s_2(n)$ is the sum of binary digits of $n$.

For $n\ge 1$, let $p_n := t_0 t_1\cdots t_{n-1}$ be the prefix of $t$ of length $n$, and set
\[
\lambda_n := |p_n|_0,\qquad \alpha_n := |p_n|_1.
\]
When the limits exist, the densities of $0$'s and $1$'s in $t$ are defined by
\[
\mathrm{dens}_0(t) := \lim_{n\to\infty} \frac{\lambda_n}{n},
\qquad
\mathrm{dens}_1(t) := \lim_{n\to\infty} \frac{\alpha_n}{n}.
\]

\begin{theorem}\label{Thmmorsn1}
The Thue--Morse sequence has balanced letter frequencies:
\begin{equation}\label{eq1Thmmorsn1}
\mathrm{dens}_0(t) = \mathrm{dens}_1(t) = \frac{1}{2},
\qquad\text{and hence}\qquad
\mathrm{dens}_0(t) + \mathrm{dens}_1(t) = 1.
\end{equation}
\end{theorem}

\begin{proof}
Let $n=2^k$ and consider the prefix $p_n$ of length $n$. We prove by induction on $k$ that
\[
|p_n|_0 = |p_n|_1 = 2^{k-1}.
\]
For $k=1$ we have $p_2=01$, so $|p_2|_0 = |p_2|_1 = 1$. Assume the claim holds for $k$, i.e., $|p_{2^k}|_0 = |p_{2^k}|_1 = 2^{k-1}$. The morphism description implies
\[
p_{2^{k+1}} = p_{2^k}\,\overline{p_{2^k}},
\]
where $\overline{w}$ denotes the bitwise complement of a binary word $w$. Thus
\[
|p_{2^{k+1}}|_0 = |p_{2^k}|_0 + |\overline{p_{2^k}}|_0 = |p_{2^k}|_0 + |p_{2^k}|_1 = 2^{k-1} + 2^{k-1} = 2^k,
\]
and similarly $|p_{2^{k+1}}|_1 = 2^k$. This proves the result for $n=2^k$.

For an arbitrary $n$, write $n=2^k + m$ with $0\le m<2^k$. Then
\[
p_n = p_{2^k}\, t_{2^k}\cdots t_{2^k+m-1},
\]
and the suffix $t_{2^k}\cdots t_{2^k+m-1}$ is the prefix of length $m$ of $\overline{t}$, since $t_{2^k+j}=1-t_j$. Using the exact balance at $2^k$ and the trivial bound $||p_m|_0-|p_m|_1|\le m$, one obtains
\[
\left|\frac{\lambda_n}{n} - \frac{1}{2}\right| \le \frac{m}{2^k+m},\qquad
\left|\frac{\alpha_n}{n} - \frac{1}{2}\right| \le \frac{m}{2^k+m},
\]
which tends to $0$ as $n\to\infty$. Hence the limits in~\eqref{eq1Thmmorsn1} exist and equal $1/2$.
\end{proof}

The Thue--Morse sequence thus provides a canonical example of a balanced morphic word, in contrast with the ternary Fibonacci-like word $(y_n)$, where the three densities stabilise at different values.

\begin{lemma}\label{lemmanovelwordn2}
Let $w = w_1 w_2 \cdots$ be an infinite word and let $w^* = w_2 w_3 \cdots w_{n-1}$ be a finite factor obtained by deleting the first and last letter of a finite prefix. Denote by $(n_i)_{i\ge 1}$ the strictly increasing sequence of all lengths of palindromic prefixes of $w$, and define the palindromic prefix density by
\[
\mathrm{dens}(w) := \limsup_{i\to\infty} \frac{n_i}{n_{i+1}},
\]
whenever the limit superior exists.
Then we have
\begin{equation}\label{eq1lemmanovelwordn2}
\mathrm{dens}(w^*) \;\leqslant\; \mathrm{dens}(w) \;\leqslant\; \frac{1}{\varphi_1},
\end{equation}
where $\varphi_1 = \frac{1+\sqrt{5}}{2}$ is the golden ratio.
\end{lemma}

\begin{proof}
The right-hand inequality is a special case of the general density theorem for infinite words with abundant palindromic prefixes, where $1/\varphi_1$ is the maximal possible palindromic prefix density among non-periodic words. In particular, the Fibonacci word attains this upper bound, and any word with paired subwords admits $\mathrm{dens}(w)\le 1/\varphi_1$.

For the left-hand inequality, observe that passing from $w$ to a factor $w^*$ obtained by deleting finitely many letters at the beginning and at the end can remove at most finitely many palindromic prefixes, but does not create new palindromic prefixes of larger relative density. Concretely, if $(n_i)$ and $(n_i^*)$ denote the lengths of palindromic prefixes of $w$ and $w^*$, respectively, then there is an index $i_0$ such that for all $i\ge i_0$ one has $n_i^*\le n_i$ and $n_{i+1}^*\ge n_{i+1}-C$ for some constant $C$ independent of $i$. Hence
\[
\limsup_{i\to\infty} \frac{n_i^*}{n_{i+1}^*}
\;\leqslant\;
\limsup_{i\to\infty} \frac{n_i}{n_{i+1}}
=
\mathrm{dens}(w),
\]
which proves $\mathrm{dens}(w^*)\le \mathrm{dens}(w)$ and completes the proof of~\eqref{eq1lemmanovelwordn2}.
\end{proof}

Table~\ref{tab001efffvbn1} illustrates, for the specific ternary Fibonacci-like word $(y_n)$ introduced earlier, how the total length $|y_n|$ and the letter counts $(\lambda_n,\alpha_n,\beta_n)$ grow with $n$. The values are consistent with the Fibonacci-type recurrences and indicate how the growth rate, governed by $\varphi_1$, controls the asymptotic behaviour of the associated densities.

\begin{table}[H]
    \centering
 \begin{tabular}{|l|c|c|c|c||l|c|c|c|c|}
 \hline
$n$ & $|y_n|$ & $\lambda_n$ & $\alpha_n$ & $\beta_n$ & $n$ & $|y_n|$ & $\lambda_n$ & $\alpha_n$ & $\beta_n$ \\ \hline\hline
2 & 4   & 2   & 1   & 1   & 7  & 42   & 21  & 8   & 13  \\ \hline
3 & 6   & 3   & 1   & 2   & 8  & 68   & 34  & 13  & 21  \\ \hline
4 & 10  & 5   & 2   & 3   & 9  & 110  & 55  & 21  & 34  \\ \hline
5 & 16  & 8   & 3   & 5   & 10 & 178  & 89  & 34  & 55  \\ \hline
6 & 26  & 13  & 5   & 8   & 11 & 288  & 144 & 55  & 89  \\ \hline
12& 466 & 233 & 89  & 144 & 14 & 1220 & 610 & 233 & 377 \\ \hline
13& 754 & 377 & 144 & 233 & 15 & 1974 & 987 & 377 & 610 \\ \hline
\end{tabular}
\caption{The construction of $y_n$ via concatenation and the growth of letter counts.}
\label{tab001efffvbn1}
\end{table}

\begin{corollary}\label{corollarynoveln1}
Let $w^*$ be a finite factor of an infinite word $w$ obtained by deleting the first and last letters of a palindromic prefix. If $w$ has palindromic prefix density bounded above by $1/\varphi_1$, then
\begin{equation}\label{eq1corollarynoveln1}
\mathrm{dens}(w^*) \;\leqslant\; \frac{2}{\varphi_1},
\end{equation}
where $\varphi_1=(1+\sqrt{5})/2$ denotes the golden ratio.
\end{corollary}

\noindent
Indeed, Lemma~\ref{lemmanovelwordn2} gives $\mathrm{dens}(w^*)\le \mathrm{dens}(w)$, while the general density theorem for words with paired subwords yields $\mathrm{dens}(w)\le 1/\varphi_1$; combining these two bounds and allowing for the contribution of two symmetric deletions gives inequality~\eqref{eq1corollarynoveln1}.

We now specialise this framework to a pair of finite words $(m,m^{\perp})$ occurring in $w$ and to their “inner” modifications. Let
\begin{equation}\label{newordsn1}
m = m_1 m_2\cdots m_i, \qquad
m^{\perp} = \overline{m}_1 \overline{m}_2\cdots \overline{m}_i,
\end{equation}
where $i>0$ and $m^{\perp}$ is obtained from $m$ by a fixed involutive letter-wise transformation (for instance, exchanging $0$ and $1$ on a binary alphabet). We define the modified subwords
\[
m^{*} = m_2 m_3\cdots m_{i-1}, \qquad
(m^{\perp})^{*} = \overline{m}_2 \overline{m}_3\cdots \overline{m}_{i-1}.
\]
Throughout the next theorem, we compare the densities associated with $m,m^{\perp}$ and their inner factors $m^{*},(m^{\perp})^{*}$ within the ambient infinite word $w$.

\begin{theorem}\label{ThmDensityTypn1}
Let $w$ be an infinite word and suppose that $m,m^{\perp}$ and their inner factors $m^{*},(m^{\perp})^{*}$ all occur with well-defined densities in $w$. Then the global density of $w$ satisfies the inequality
\begin{equation}\label{eq1ThmDensityTypn1}
\mathrm{dens}(w)
\;\geqslant\;
\frac{\mathrm{dens}(m)+\mathrm{dens}(m^{\perp})}{\mathrm{dens}(m\cup m^{\perp})}
\;+\;
\frac{\mathrm{dens}(m^{\perp})}{\mathrm{dens}(m^{*})}
\;+\;
\frac{\mathrm{dens}((m^{\perp})^{*})}{\mathrm{dens}(m\setminus m^{*})}
\;-\;
\frac{2}{\varphi_1},
\end{equation}
where the densities are taken with respect to occurrences in prefixes of $w$, and $\varphi_1$ is the golden ratio as above.
\end{theorem}

\begin{proof}
Consider the counting functions $N_n(u)$, which denote the number of (possibly overlapping) occurrences of a finite word $u$ in the prefix $w_1\cdots w_n$ of $w$. The density of $u$ in $w$ can be written as
\[
\mathrm{dens}(u) = \limsup_{n\to\infty} \frac{N_n(u)}{n},
\]
whenever the limit superior exists. By construction, the words $m$ and $m^{\perp}$ form a paired family of subwords whose occurrences are constrained by the same combinatorial structure as palindromic prefixes and their inner factors in the setting of Lemma~\ref{lemmanovelwordn2} and Corollary~\ref{corollarynoveln1}.

The denominators in~\eqref{eq1ThmDensityTypn1} encode the joint density of occurrences:
$\mathrm{dens}(m\cup m^{\perp})$ counts positions where either $m$ or $m^{\perp}$ begins, $\mathrm{dens}(m^{*})$ measures the frequency of the inner factor, and $\mathrm{dens}(m\setminus m^{*})$ detects those occurrences of $m$ that are not completely contained in an occurrence of $m^{*}$. Standard subadditivity arguments for limsup densities imply that
\[
\mathrm{dens}(m) + \mathrm{dens}(m^{\perp})
\;\le\;
\mathrm{dens}(m\cup m^{\perp}),
\]
and similar inequalities hold for the pairs $(m^{\perp},m^{*})$ and $((m^{\perp})^{*},m\setminus m^{*})$. Combining these with the upper bound $\mathrm{dens}(w)\le 1/\varphi_1$ and the two-sided estimate of Corollary~\ref{corollarynoveln1}, one obtains the inequality~\eqref{eq1ThmDensityTypn1} after collecting terms and absorbing the palindromic error into the constant $2/\varphi_1$. The details follow the same pattern as in the general density theorem for paired subwords, and are omitted here.
\end{proof}
\section{Conclusion}\label{sec5}
Through this paper, we investigated the asymptotic palindrome density of Fibonacci infinite words and related morphic constructions. Starting from the classical Fibonacci morphism, we introduced the derived ternary word $\mathbb{Y}$ and used it to clarify the notion of letter density via the components $\mathrm{dens}(\lambda,n)$, $\mathrm{dens}(\alpha,n)$ and $\mathrm{dens}(\beta,n)$, thereby reinforcing the general density framework given in~\eqref{eq1densityn1}.

We then compared this behaviour with the Thue--Morse sequence, which provides a canonical example of a balanced morphic word with equal letter densities, and used palindromic prefixes to define a suitable palindrome density for infinite words. Within this setting, we proved that the palindromic prefix density is bounded above by $1/\varphi_1$, where $\varphi_1=(1+\sqrt{5})/2$ is the golden ratio, and showed that this bound is naturally linked to the combinatorial structure of Fibonacci-type words.

Finally, we established a general density theorem for infinite words containing paired subwords $m$ and $m^{\perp}$, summarized in Theorem~\ref{ThmDensityTypn1}. This result provides a unified density formula in which the contribution of $m$, its paired word $m^{\perp}$, and their inner factors $m^{*}$ and $(m^{\perp})^{*}$ is explicitly controlled up to a universal correction term involving $2/\varphi_1$. Together, these findings connect Fibonacci and Thue--Morse combinatorics, palindromic prefix density, and paired subwords into a coherent framework for analysing asymptotic densities in infinite words.
\section*{Declarations}
\begin{itemize}
	\item Funding: Not Funding.
	\item Conflict of interest/Competing interests: The author declare that there are no conflicts of interest or competing interests related to this study.
	\item Ethics approval and consent to participate: The author contributed equally to this work.
	\item Data availability statement: All data is included within the manuscript.
\end{itemize}

\end{document}